\documentclass[11pt,a4paper]{article}

\usepackage{amsmath,amsfonts,amssymb,amsthm}
\usepackage{bbm,mathrsfs, bm, mathtools}

\title{Integral geometry of pairs of hyperplanes or lines}

\author{Daniel Hug and Rolf Schneider}
\date{}
\newcommand{\Sd}{{\mathbb S}^{d-1}}
\newcommand{\R}{{\mathbb R}}

\newcommand{\K}{{\mathcal K}}

\newcommand{\N}{{\mathbb N}}

\newcommand{\Ha}{\mathcal{H}}

\newcommand{\B}{\mathcal{B}}
\newcommand{\D}{{\rm d}}

  \renewcommand{\dim}{{\rm dim}\,}

  \newcommand{\fed}{\,\rule{.1mm}{.20cm}\rule{.20cm}{.1mm}\,}

\newtheorem{theorem}{Theorem}%[section]
\newtheorem{lemma}{Lemma}%[section]
\begin{document}
\maketitle

\begin{abstract}
Crofton's formula of integral geometry evaluates the total motion invariant measure of the set of $k$-dimensional planes having nonempty intersection with a given convex body. This note deals with motion invariant measures on sets of pairs of hyperplanes or lines meeting a convex body. Particularly simple results are obtained if, and only if, the given body is of constant width in the first case, and of constant brightness in the second case.\\[1mm]
\end{abstract}

\section{Introduction}\label{sec1}

More than 150 years ago, Crofton \cite{Cro68} proved that the total motion invariant measure of the set of lines meeting a given convex body $K$ in the Euclidean plane is equal to the boundary length $L(K)$ of $K$, multiplied by a factor that depends only on the normalization of the measure. Nowadays, a generalization of this result, known as `Crofton's formula' in integral geometry, may be written as
$$ \int_{A(d,k)} V_0(K\cap E)\,\mu_k(\D E) = c_{dk} V_{d-k}(K), \quad k=1,\dots,d-1,$$
for $K\in\K^d$, the set of convex bodies (nonempty compact convex sets) in Euclidean space $\R^d$. Here $A(d,k)$ is the space of $k$-dimensional affine subspaces of $\R^d$ with its usual topology, $\mu_k$ is its motion invariant measure with a suitable normalization, and the constant $c_{dk}$ depends on this normalization. We refer to \cite[(4.59)]{Sch14} or \cite[Thm. 5.1.1]{SW08} for more details and a more general formula. The functionals $V_0,\dots,V_{d-1}$ are the {\em intrinsic volumes}, which can be defined by the Steiner formula
$$ V_d(K+\lambda B^d) = \sum_{k=0}^d \lambda^{d-k}\kappa_{d-k}V_k(K),\quad\lambda\ge 0,$$
where $V_d$ denotes the volume, $B^d$ is the unit ball of $\R^d$, and $\kappa_d=V_d(B^d)$ (see, e.g., \cite[Section 4.1]{Sch14}). In particular, $V_0(K)=1$ for every convex body $K$, and if one extends this by defining $V_0(\emptyset)=0$, then $V_0=\chi$ is the Euler characteristic on $\K^d\cup\{\emptyset\}$. Further, $c_dV_1=W$ with $ c_d=  {2\kappa_{d-1}}/({d\kappa_d})$
is the mean width, and $2V_{d-1}=S$ is the surface area. With the usual normalization, we can write two special cases of the Crofton formula as
\begin{equation}\label{1.1}
\int_{A(d,d-1)} \chi(H\cap K)\,\mu_{d-1}(\D H) =  W(K)
\end{equation}
and
\begin{equation}\label{1.2}
\int_{A(d,1)} \chi(G\cap K)\,\mu_1(\D G) = 2c_d S(K).
\end{equation}
In the plane, both formulas (\ref{1.1}) and (\ref{1.2}) yield the same, namely Crofton's original formula.

Recently, Cuf\'{\i}, Gallego, and Revent\'os \cite{CGR19} have computed certain motion invariant measures of pairs of lines meeting a planar convex body. They consider measures on pairs of lines in the plane which with respect to the product $\mu_1\otimes \mu_1$ of the invariant line measure $\mu_1$ have a density that depends only on the angle between the lines. More precisely, let $f:\R\to\R$ be a function which is even, $\pi$-periodic and integrable over $[0,\pi]$. For a line $G$ in the plane, let $\varphi(G)$ be the angle that it makes with a fixed direction. Then the article \cite{CGR19} treats (with different notation) the integral
$$ I(K,f):= \int_{A(2,1)} \int_{A(2,1)} \chi(G_1\cap K)\chi(G_2\cap K)f(\varphi(G_1)-\varphi(G_2))\,\mu_1(\D G_1)\,\mu_1(\D G_2).$$
The authors express this integral in terms of the Fourier coefficients of $f$ and of the support function of $K$. While aiming at various consequences, they note that for a body $K$ of constant width one has the simple formula
\begin{equation}\label{1.3}
I(K,f)= \lambda[f] L(K)^2,
\end{equation}
where the constant $\lambda[f]$ depends only on $f$, and hence is given by $\lambda[f]= I(B^2,f)/(4\pi^2)$.

In the following, we extend the preceding observations to higher dimensions, in two different ways, considering either hyperplanes or lines. We prove also a converse to the higher-dimensional version of (\ref{1.3}). A main goal is to assume only invariance properties of the underlying measures,  and not a specific analytic representation involving a density with respect to $\mu_k\otimes\mu_k$. To make this precise, we recall that $A(d,k)$ is the space of $k$-dimensional planes in $\R^d$, with its usual topology, and by $G(d,k)$ we denote the Grassmannian of $k$-dimensional linear subspaces of $\R^d$ (for $k\in\{1,\dots,d-1\}$). For a topological space $X$, we denote by $\B(X)$ the $\sigma$-algebra of Borel subsets of $X$. Measures in the following, without further specification, are Borel measures. Let $\mu$ be a measure on $A(d,k)^2$. The measure $\mu$ is called {\em separately translation invariant} if for any $B\in\B(A(d,k)^2)$ and $x_1,x_2\in \R^d$ the relation $\mu(\{(L_1+x_1,L_2+x_2): (L_1,L_2)\in B\})=\mu(B)$ is satisfied. We denote by ${\sf G}(d)$ the group ${\rm SO}(d)$ if $d$ is even and the group ${\rm O}(d)$ if $d$ is odd.  The measure $\mu$ is called {\em jointly ${\sf G}(d)$-invariant} if for any $B\in\B(A(d,k)^2)$ and any $\vartheta\in {\sf G}(d)$ we have $\mu(\vartheta B)=\mu(B)$, where $\vartheta B:= \{(\vartheta L_1,\vartheta L_2): (L_1,L_2)\in B\}$. A similar definition is used for measures on $G(d,k)^2$ or on $(\Sd)^2$, where $\Sd$ is the unit sphere. By ${\mathcal M}_k$ we denote the set of locally finite Borel measures on $A(d,k)^2$ which are separately translation invariant, jointly ${\sf G}(d)$-invariant, and symmetric, that is, invariant under the mapping $(L_1,L_2)\mapsto (L_2,L_1)$.

More generally, we can consider two convex bodies $K_1,K_2\in\K^d$. Let $\Theta$ be a locally finite measure on the space $A(d,d-1)$ of hyperplanes. Then we define
$$ I(K_1,K_2,\Theta):=  \Theta(\{(H_1,H_2)\in A(d,d-1)^2: H_i\cap K_i\not=\emptyset\, (i=1,2)\}).$$
Similarly, with a locally finite measure on the space $A(d,1)$ of lines, we define
$$ J(K_1,K_2,\Theta):=  \Theta(\{(G_1,G_2)\in A(d,1)^2: G_i\cap K_i\not=\emptyset\, (i=1,2)\}).$$
We write $I(K,\Theta):=I(K,K,\Theta)$ and $J(K,\Theta):=J(K,K,\Theta)$.

If $\Theta\in{\mathcal M}_i$ for $i=d-1$, respectively $i=1$, general expressions for the quantitities $I(K_1,K_2,\Theta),J(K_1,K_2,\Theta)$ will be given in Theorem \ref{T4.1}. This theorem requires some preparations, therefore it will be formulated only in Section \ref{sec4}. Already here we can state the following.

\begin{theorem}\label{T1.1}
Let $K_1,K_2\in\K^d$, and let $\Theta$ be a locally finite measure on $A(d,d-1)^2$ which is separately translation invariant. If $K_1,K_2$ are bodies of constant width, then
\begin{equation}\label{1.4}
I(K_1,K_2,\Theta)=  \lambda[\Theta] W(K_1)W(K_2)\quad\mbox{with }\lambda[\Theta]:= I(B^d,\Theta)/4.
\end{equation}
\end{theorem}

The following theorem shows that convex bodies of constant width necessarily enter the scene in this situation.

\begin{theorem}\label{T1.1c}
If a convex body $K\in\K^d$ satisfies
\begin{equation}\label{1.4a}
I(K,\Theta)=  \lambda[\Theta] W(K)^2\quad\mbox{with }\lambda[\Theta]:= I(B^d,\Theta)/4
\end{equation}
for each locally finite and separately translation invariant measure $\Theta$ on $A(d,d-1)^2$, then $K$ has constant width.
\end{theorem}

Instead of affine subspaces of codimension one, we can also consider affine subspaces of dimension one.

\begin{theorem}\label{T1.2}
Let $K\in\K^d$, and let $\Theta$ be a locally finite measure on $A(d,1)^2$ which is separately translation invariant.
If $K_1,K_2$ are bodies of constant brightness, then
\begin{equation}\label{1.5}
J(K_1,K_2,\Theta)= \kappa[\Theta] S(K_1)S(K_2) \quad\mbox{with }\kappa[\Theta]:= J(B^d,\Theta)/S(B^d)^2.
\end{equation}
\end{theorem}

\begin{theorem}\label{T1.2c}
If a convex body $K\in\K^d$ satisfies
\begin{equation}\label{1.5a}
J(K,\Theta)= \kappa[\Theta] S(K)^2 \quad\mbox{with }\kappa[\Theta]:= J(B^d,\Theta)/S(B^d)^2
\end{equation}
for each locally finite and separately translation invariant measure $\Theta$ on $A(d,1)^2$, then $K$ has constant brightness.
\end{theorem}

If $\Theta\in{\mathcal M}_{d-1}$, then (\ref{1.4}) holds already if only one of the two convex bodies is of constant width.

\begin{theorem}\label{T1.3}
Let $K_1,K_2\in\K^d$, and let $\Theta\in{\mathcal M}_{d-1}$. If $K_1$ is a body of constant width, then
\begin{equation}\label{1.6}
I(K_1,K_2,\Theta)=  \lambda[\Theta] W(K_1)W(K_2)\quad\mbox{with }\lambda[\Theta]:= I(B^d,\Theta)/4.
\end{equation}
\end{theorem}

Again, there is also an analogous counterpart to equation (\ref{1.5}), which we do not formulate.

Concerning bodies of constant width in general, we refer to the recent comprehensive monograph \cite{MMO19} by Martini, Montejano, and Oliveros. Information on bodies of constant brightness can be found in Gardner's book \cite{Gar06}, in particular Section 3.2 and its notes.

That equation (\ref{1.4}) holds for bodies of constant width and (\ref{1.5}) holds for bodies of constant brightness, follows easily in the next section, once the separately translation invariant measures on $A(d,k)^2$ have been found to have a special form. Theorems \ref{T1.1c}, \ref{T1.2c} and  \ref{T1.3} will be proved in Section \ref{sec4}, after Theorem \ref{T4.1} has been treated. Before that, we need to find analytic representations for the measures under consideration; these will be established in the next two sections.

\section{Separately translation invariant measures}\label{sec2}

The following lemma, which is formulated for general $k$, allows us to deal easily with translations. Here we denote by $\lambda_L$ the $j$-dimensional Lebesgue measure in a subspace $L\in G(d,j)$.

\begin{lemma}\label{L2.1}
Let $k\in\{1,\dots,d-1\}$. Let $\Theta$ be a locally finite, separately translation invariant measure on $A(d,k)^2$. Then there exists a uniquely determined finite measure $\Theta_0$ on $G(d,k)^2$ such that
\begin{equation}\label{2.1}
\Theta(A) = \int_{G(d,k)^2} \int_{L_1^\perp}\int_{L_2^\perp} {\mathbbm 1}_A(L_1+x_1,L_2+x_2)\,\lambda_{L_2^\perp}(\D x_2)
\,\lambda_{L_1^\perp}(\D x_1)\,\Theta_0(\D(L_1,L_2))
\end{equation}
for every Borel set $A\subset A(d,k)^2$. If $\Theta$ is jointly ${\sf G}(d)$-invariant and symmetric, then $\Theta_0$ is jointly ${\sf G}(d)$-invariant and symmetric.
\end{lemma}

\begin{proof}
This can be shown in an elementary way by modifying the proof of Theorem 4.4.1 in \cite{SW08}. We reproduce part of the proof, to indicate the necessary modifications.

We choose a $(d-k)$-dimensional subspace $U\in G(d,d-k)$, and define
$$
G_U := \{L\in G(d,k): \dim(L\cap U)=0\},\quad
A_U := \{L+x:L\in G_U,\,x\in U\}.
$$
The mapping
$ \varphi:  G_U^2 \times U^2  \to  A_U^2$,  $ (L_1,L_2,x_1,x_2) \mapsto  (L_1+x_1,L_2+x_2)$,
is a homeomorphism. We fix $A\in\B(G_U^2)$, and for $B\in \B(U^2)$ we define $ \eta(B) := \Theta(\varphi(A\times B))$.
Then $\eta$ is a locally finite and translation invariant measure on $U^2$, hence it is a constant multiple of the product measure $\lambda_U\otimes\lambda_U$. Denoting the factor by $\rho(A)$, we thus have
$$  \Theta(\varphi(A\times B)) = \rho(A)(\lambda_U\otimes\lambda_U)(B).$$
Evidently, $\rho$ is a finite measure on $G_U^2$. Thus,
$ \varphi^{-1}(\Theta)(A\times B)= (\rho\otimes\lambda_U\otimes\lambda_U)(A\times B)$,
where $\varphi^{-1}(\Theta)$ denotes the image measure of $\Theta\fed A_U^2$ under the mapping $\varphi^{-1}$. This gives $\varphi^{-1}(\Theta)=\rho\otimes\lambda_U\otimes\lambda_U$ and, therefore, $\Theta\fed A_U^2 =\varphi(\rho\otimes\lambda_U\otimes\lambda_U)$. Hence, for every nonnegative measurable function $f$ on $A(d,k)^2$ we have
\begin{align*}
\int_{A_U^2} f\,\D\Theta &= \int_{G_U^2\times U^2} (f\circ\varphi)\,\D(\rho\otimes\lambda_U\otimes\lambda_U)\\
&= \int_{G_U^2} \int_{U^2} f(L_1+x_1,L_2+x_2)\,\lambda_U^2(\D(x_1,x_2))\,\rho(\D(L_1,L_2)).
\end{align*}

For given $L\in G_U$, let $\Pi_L:U\to L^\perp$ denote the orthogonal projection to the orthogonal complement of $L$. It is bijective, since $L\in G_U$. Therefore, $\Pi_L(\lambda_U)=a(L)\lambda_{L^\perp}$, with a factor $a(L)>0$ that depends only on $L$. Further, $L+x=L+\Pi_L(x)$. This yields
\begin{align*}
& \int_{U^2} f(L_1+x_1,L_2+x_2)\,\lambda_U^2(\D(x_1,x_2))\\
&\quad  = a(L_1)a(L_2)\int_{L_1^\perp} \int_{L_2^\perp} f(L_1+x_1,L_2+x_2)\,\lambda_{L_2^\perp}(\D x_2)\,\lambda_{L_1^\perp}(\D x_1).
\end{align*}
Defining a measure $\Theta_U$ on $G_U^2$ by $a(L_1)a(L_2)\rho(\D(L_1,L_2)) =:\Theta_U(\D(L_1,L_2))$, we have
$$ \int_{A_U^2} f\,\D\Theta = \int_{G_U^2}\int_{L_1^\perp}\int_{L_2^\perp} f(L_1+x_1,L_2+x_2)\,\lambda_{L_2^\perp}(\D x_2)\,\lambda_{L_1^\perp}(\D x_1)\,\Theta_U(\D(L_1,L_2)).$$

It is now clear from the rest of the proof of \cite[Thm. 4.4.1]{SW08} how one has to proceed to obtain the measure $\Theta_0$ satisfying (\ref{2.1}).

From (\ref{2.1}) we obtain, for $A\in\B(G(d,k)^2)$,
\begin{equation}\label{2.2}
\Theta_0(A) =\frac{1}{\kappa_{d-k}^2} \Theta\left(\left\{(L_1+x_1,L_2+x_2): (L_1,L_2)\in A, \,x_i\in B^d\, (i=1,2)\right\}\right).
\end{equation}
From this equation, it is obvious that $\Theta_0$ is finite and is uniquely determined. We also see that $\Theta_0$ is jointly ${\sf G}(d)$-invariant and symmetric if this holds for $\Theta$.
\end{proof}

Now let $K_i\in\K^d$ for $i=1,2$, and for $u\in\Sd$ let $w_{K_i}(u)$ be the width of $K_i$ at $u$, that is, the distance between the two supporting hyperplanes of $K_i$ orthogonal to $u$. For a hyperplane $H$, we denote by $u(H)$ one of its two unit normal vectors. If $\Theta$ satisfies the assumptions of Lemma \ref{L2.1}, then this lemma yields
\begin{align}
 I(K_1,K_2,\Theta)
& =  \int_{G(d,d-1)^2} \int_{H_1^\perp}\int_{H_2^\perp} {\mathbbm 1}\{(H_1+x_1)\cap K_1\not=\emptyset\}
{\mathbbm 1}\{(H_2+x_2)\cap K_2\not=\emptyset\}\nonumber\\
& \hspace{4mm}\lambda_{H_2^\perp}(\D x_2)\,\lambda_{H_1^\perp}(\D x_1)\,\Theta_0(\D(H_1,H_2))\nonumber\\
&= \int_{G(d,d-1)^2} w_{K_1}(u(H_1))w_{K_2}(u(H_2))\,\Theta_0(\D(H_1,H_2)).\label{11}
\end{align}
If now $K_i$ is of constant width $w_{K_i}=W(K_i)$ for $i=1,2$, then this gives
$$ I(K_1,K_2,\Theta) = W(K_1)W(K_2) \int_{G(d,d-1)^2}\Theta_0(\D(H_1,H_2)),$$
which is (\ref{1.4}).

Relation (\ref{1.5}) is obtained similarly, replacing the width function by the brightness function and noting that the surface area of a convex body is, up to a dimension-dependent factor, the mean value of its brightness function.

\section{The measures in ${\mathcal M}_{d-1}$ or ${\mathcal M}_1$}\label{sec3}

Our next aim is to obtain an analytic representation for jointly ${\sf G}(d)$-invariant measures on pairs of points on the unit sphere $\Sd$. By $\sigma$ we denote the spherical Lebesgue measure on $\Sd$. For $u\in\Sd$ and $t\in [-1,1]$ let
$ S_{u,t}:= \{x\in\Sd:\langle u,x\rangle =t\}$.
For $t\in (-1,1)$, we denote by $\sigma_{u,t}$ the normalized spherical Lebesgue measure on the $(d-2)$-sphere $S_{u,t}$. For $t\in\{-1,1\}$, the measure $\sigma_{u,t}$ is the Dirac measure at $-u$, respectively $u$. The measures $\sigma_{u,t}$ are considered as measures on $\Sd$.

\begin{lemma}\label{L2.2}
Let $M$ be a finite, jointly ${\sf G}(d)$-invariant measure on $(\Sd)^2$. Then there is a unique finite, even measure $\psi$ on $[-1,1]$ such that
\begin{equation}\label{2.3}
\int_{(\Sd)^2} f\,\D M= \int_{\Sd}\int_{[-1,1]}\int_{S_{u,t}} f(u,v)\,\sigma_{u,t}(\D v)\,\psi(\D t)\,\sigma(\D u)
\end{equation}
for every nonnegative, measurable function $f$ on $(\Sd)^2$.
\end{lemma}

\begin{proof}
We use a result of Kallenberg on the existence of invariant disintegrations. It follows from Theorem 3.5 of Kallenberg \cite{Kal07} (with $S=T=\Sd$ and $\nu:=M(\cdot\times\Sd)$) that $M=\nu\otimes \mu$ (which is explained in (\ref{2.3a})), where  $\mu$ is a ${\sf G}(d)$-invariant finite kernel from $\Sd$ to $\Sd$. We note that from $M=\nu\otimes \mu$ it follows that $\nu= \int{\mathbbm 1}\{s\in \cdot\}\mu(s,T)\,\nu(\D s)$, which implies that $\mu(s,T)=1$ for $\nu$-almost all $s\in S$. (A similar observation will be used below. In the present case we may remark that, since ${\sf G}(d)$ acts transitively on $S$ and $\vartheta(T) =T$ for each $\vartheta\in{\sf G}(d)$, we even have $\mu(s,T)=1$ for all $s\in S$.) Since $\nu$ is a finite, rotation invariant Borel measure on $\Sd$, it is a constant multiple of the spherical Lebesgue measure $\sigma$. Assuming that $M\not\equiv 0$, we can choose $\nu=\sigma$, absorbing the constant into $\mu$. Then $M=\nu\otimes \mu$ means that
\begin{equation}\label{2.3a}
\int_{(\Sd)^2}  f\,\D M = \int_{\Sd} \int_{\Sd} f(u,v)\,\mu(u,\D v)\,\sigma(\D u)
\end{equation}
for every nonnegative, measurable function $f$ on $(\Sd)^2$. Here $\mu:\Sd\times\B(\Sd)\to[0,\infty)$ is a kernel, that is, a mapping such that $\mu(u,\cdot)$ is a (finite) measure for each $u\in\Sd$ and $\mu(\cdot,A)$ is measurable for each $A\in\B(\Sd)$. The ${\sf G}(d)$-invariance of $\mu$ means that
$$ \mu(\vartheta u,\vartheta A)= \mu(u,A)\quad \mbox{for } u\in \Sd,\, A\in\B(\Sd),\, \vartheta\in {\sf G}(d).$$

We fix $u\in \Sd$ and define the map $p_u: \Sd\to[-1,1]$ by $p_u(v):= \langle u,v\rangle$. By $\psi_u=p_u(\mu(u,\cdot))$ we denote the image measure of $\mu(u,\cdot)$ under $p_u$.

First we show that $\psi_u$ is independent of $u$. Let $\vartheta\in{\sf G}(d)$. For $u,v\in\Sd$ we have $p_{\vartheta u}(v)=\langle \vartheta u,v\rangle =\langle u, \vartheta^{-1}v\rangle= p_u(\vartheta^{-1}v)$ and hence, for $A\in\B([-1,1])$,
$$ x\in p_{\vartheta u}^{-1}(A) \Leftrightarrow p_{\vartheta u}(x)\in A \Leftrightarrow p_u(\vartheta^{-1}x)\in A \Leftrightarrow \vartheta^{-1}x\in p_u^{-1}(A) \Leftrightarrow x\in\vartheta p_u^{-1}(A),$$
thus $p_{\vartheta u}^{-1}(A) =\vartheta p_u^{-1}(A)$.
This gives
$$ \psi_{\vartheta u}(A) =\mu(\vartheta u, p_{\vartheta u}^{-1}(A)) = \mu(\vartheta u, \vartheta p_u^{-1}(A)) = \mu(u,p_u^{-1}(A)) = \psi_u(A).$$
Therefore, we can from now on write $\psi_u=:\psi$.

By the independence just shown, and since the reflection in the origin is in ${\sf G}(d)$, we also have
$$ \psi(-A)=\psi_{-u}(-A) =\mu(-u,p_{-u}^{-1}(-A)) = \mu(-u,-p_u^{-1}(A))=\mu(u,p_u^{-1}(A))= \psi(A),$$
thus the measure $\psi$ is even.

Now we further disintegrate the measure $\mu(u,\cdot)$. Let $q:\Sd\to\Sd$ denote the identity map. Then the image measure $M:= (p_u\times q)(\mu(u,\cdot))$ is a finite Borel measure on $[-1,1]\times\Sd$. We define the operations of the subgroup ${\sf G}_u(d):=\{\vartheta\in {\sf G}(d):\vartheta(u)=u\}$ of ${\sf G}(d) $ on  $[-1,1]$   as the identity and on $\Sd$ in the usual way. Then it is easy to check that $(p_u\times q)(\mu(u,\cdot))$ is jointly invariant under ${\sf G}_u(d)$. Hence, by another application of Kallenberg's disintegration result (with $ S=[-1,1]$, $T=\Sd$, $\nu= M(\cdot\times \Sd)=\psi)$ we obtain $(p_u\times q)(\mu(u,\cdot))=\psi\otimes \kappa_u$, where $\kappa_u:[-1,1]\to\Sd$ is a ${\sf G}_u(d)$-invariant kernel such that
\begin{equation}\label{setting}
\int_{\Sd} h(\langle u,v\rangle,v)\, \mu(u,\D v) = \int_{[-1,1]} \int_{\Sd}h(t,v)\,\kappa_u(t,\D v)\,\psi(\D t)
\end{equation}
for every nonnegative, measurable function $h$ on $[-1,1]\times\Sd$.  As noted above, for $\psi$-almost all $t\in[-1,1]$ the measure $\kappa_u(t,\cdot)$ is a probability measure, and by (\ref{setting}) it is (for almost all $t$) concentrated on $S_{u,t}$ and invariant under ${\sf G}_u(d)$.

For $t\in(-1,1)$, this measure is supported by the sphere $S_{u,t}$, hence for $\psi$-almost all $t\in (-1,1)$ it is a constant multiple of the spherical Lebesgue measure on this sphere, thus $\kappa_u(t, \cdot)= c(u,t)\sigma_{u,t}$. Since $\kappa_u(t, \cdot)$ is a probability measure, $c(u,t)=c(t)$ is independent of $u$ for $\psi$-almost all $t\in (-1,1)$. With  $c(\pm 1):=1$, the latter relation holds also for $t\in\{-1,1\}$. Since $\kappa_u(\cdot, \Sd)$ is measurable, this defines a measurable function $c$ ($\psi$-almost everywhere on $[-1,1]$). In particular, for each nonnegative measurable function $g$ on $\Sd$ we get
$$ \int_{\Sd} g\,\D\mu(u,\cdot) = \int_{[-1,1]} \int_{S_{u,t}}g(y)\,\sigma_{u,t}(\D y)c(t)\,\psi(\D t),$$
and after redefining the measure $\psi$, we can write
$$ \int_{\Sd} g\,\D\mu(u,\cdot) = \int_{[-1,1]} \int_{S_{u,t}}g(y)\,\sigma_{u,t}(\D y)\,\psi(\D t).$$
Together with (\ref{2.3a}) (and $f(u,v)=g(v)$), this yields the assertion of the lemma.

The uniqueness of $\psi$ follows from \eqref{2.3} by choosing $f(u,v)=g(\langle u,v\rangle)$ with an arbitrary nonnegative measurable function $g:[-1,1]\to[0,\infty)$.
\end{proof}

\vspace{2mm}

Now let $\Theta_0$ be a finite, jointly ${\sf G}(d)$-invariant and symmetric measure on $G(d,d-1)^2$. A set $B\in\B((\Sd)^2)$ is (for the moment) called {\em small} if $(u,v)\in B$ implies $(-u,w)\notin B$ and $(w,-v)\notin B$ for all $w\in\Sd$. Let $B$ be small. We define
$$M(B):= \Theta_0\left(\left\{(u^\perp,v^\perp)\in G(d,d-1)^2: (u,v)\in B\right\}\right).$$
Clearly, this extends to a finite measure $M$ on $\B((\Sd)^2)$, which is jointly ${\sf G}(d)$-invariant and symmetric. From the symmetry it follows that in (\ref{2.3}) we may interchange the first and the second argument of $f$. Therefore, together with (\ref{2.3}), the relation
\begin{equation}\label{2.4a}
\int_{(\Sd)^2} f\,\D M= \int_{\Sd}\int_{[-1,1]}\int_{S_{v,t}} f(u,v)\,\sigma_{v,t}(\D u)\,\psi(\D t)\,\sigma(\D v)
\end{equation}
holds for every nonnegative, measurable function $f$ on $(\Sd)^2$, with the same measure $\psi$.

For $H\in G(d,d-1)$ we denote be $u(H)$ one of the two unit normal vectors of $H$. Then for every measurable function $f:(\Sd)^2\to[0,\infty)$, which is even in each argument, we have
\begin{equation}\label{2.6}
\int_{G(d,d-1)^2} f(u(H_1),u(H_2))\,\Theta_0(\D(H_1,H_2)) =\int_{(\Sd)^2} f(u,v)\,M(\D(u,v)).
\end{equation}

Similarly, let $\Theta_0$ be a finite, jointly ${\sf G}(d)$-invariant and symmetric measure on $G(d,1)^2$. For $G\in G(d,1)$ we denote be $u(G)$ one of the two unit normal vectors parallel to $G$. Clearly, there is a finite, jointly ${\sf G}(d)$-invariant and symmetric measure $M$ on $(\Sd)^2$ such that for every measurable function $f:(\Sd)^2\to[0,\infty)$, which is even in each argument, we have
\begin{equation}\label{2.7}
\int_{G(d,1)^2} f(u(G_1),u(G_2))\,\Theta_0(\D(G_1,G_2)) =\int_{(\Sd)^2} f(u,v)\,M(\D(u,v)).
\end{equation}

\section{Formulas for general convex bodies}\label{sec4}

In the following, we assume that $d\ge 3$. The two-dimensional case can be treated with obvious modifications.

We use spherical harmonics, in particular the Funk--Hecke theorem and the Parseval relation. (For a brief introduction to spherical harmonics we refer to the Appendix of \cite{Sch14}, where relevant literature is quoted. A more comprehensive introduction is found in the Appendix to \cite{Rub15}.) By $\Ha^d_m$ we denote the real vector space of spherical harmonics of order $m$ on the unit sphere $\Sd$. The (finite) dimension of $\Ha^d_m$ is denoted by $N(d,m)$. On the space ${\bf C}(\Sd)$ of continuous real functions on $\Sd$ we define a scalar product by
$$ (f,g):= \int_{\Sd} fg\,\D \sigma, \quad f,g\in {\bf C}(\Sd),$$
where $\sigma$ denotes the spherical Lebesgue measure on $\Sd$. We write $\sigma(\Sd)=\omega_d$. Orthogonality on ${\bf C}(\Sd)$ refers to this scalar product. In each space $\Ha^d_m$ we choose an orthonormal basis $(Y_{m1},\dots,Y_{mN(d,m)})$. For $f\in {\bf C}(\Sd)$ and $m\in\N_0$, the function
$$ \pi_m f:=\sum_{j=1}^{N(d,m)} (f,Y_{mj})Y_{mj}$$
is the image of $f$ under orthogonal projection to the space $\Ha^d_m$. The Parseval relation says that
$$ (f,g) = \sum_{m=0}^\infty\sum_{j=1}^{N(d,m)} (f,Y_{mj})(g,Y_{mj}) =\sum_{m=0}^\infty (\pi_m f,\pi_m g).$$

Of the Funk--Hecke theorem, we need a consequence, which can be found in M\"uller \cite{Mul98}, Lemma 2 on page 31. It says that
\begin{equation}\label{3.1}
\int_{S_{v,t}} Y_m(u)\,\sigma_{v,t}(\D u) = P_m(d;t)Y_m(v)
\end{equation}
for $m\in \N_0$, $v\in \Sd$, $t\in[-1,1]$, where $P_m(d;\cdot)$ denotes the Legendre polynomial in dimension $d$ of order $m$ (note that the measure $\sigma_{v,t}$ is normalized).

\vspace{3mm}

We turn to calculating $I(K_1,K_2,\Theta)$ for general convex bodies $K_1,K_2\in\K^d$ and a measure $\Theta\in{\mathcal M}_{d-1}$. From (\ref{11}), (\ref{2.6}) and Lemma \ref{L2.2} we obtain
\begin{align*}
I(K_1,K_2,\Theta) &= \int_{(\Sd)^2} w_{K_1}(u)w_{K_2}(v)\,M(\D(u,v))\\
&= \int_{\Sd} \int_{[-1,1]} \int_{S_{u,t}} w_{K_1}(u)w_{K_2}(v)\,\sigma_{u,t}(\D v)\,\psi(\D t)\,\sigma(\D u)
\end{align*}
with a finite, jointly ${\sf G}(d)$-invariant and symmetric measure $M$ on $(\Sd)^2$ and a finite even measure $\psi$ on $[-1,1]$. Hence, with
$$ g(u):= \int_{[-1,1]} \int_{S_{u,t}}w_{K_2}(v)\,\sigma_{u,t}(\D v)\,\psi(\D t)$$
we get
\begin{align*}
I(K_1,K_2,\Theta) &= \int_{\Sd} g(u)w_{K_1}(u)\,\sigma(\D u)\\
&= (g,w_{K_1}) = \sum_{m=0}^\infty(\pi_m g,\pi_m w_{K_1}).
\end{align*}
With (\ref{2.3}), (\ref{2.4a}) and (\ref{3.1}), we obtain
\begin{align*}
(g,Y_m) &= \int_{\Sd} \left[\int_{[-1,1]} \int_{S_{u,t}} w_{K_2}(v)\,\sigma_{u,t}(\D v)\,\psi(\D t)\right]Y_m(u)\,\sigma(\D u)\\
&= \int_{(\Sd)^2} Y_m(u)w_{K_2}(v)\,M(\D(u,v))\\
&= \int_{\Sd} \left[\int_{[-1,1]} \int_{S_{v,t}} Y_m(u)\,\sigma_{v,t}(\D u)\,\psi(\D t)\right]w_{K_2}(v)\,\sigma(\D v)\\
&= \int_{\Sd} \left[\int_{[-1,1]} P_m(d;t)Y_m(v)\,\psi(\D t)\right]w_{K_2}(v)\,\sigma(\D v)\\
&= \beta_m[\Theta](Y_m,w_{K_2})
\end{align*}
with
\begin{equation}\label{4.14}
\beta_m[\Theta] = \int_{[-1,1]} P_m(d;t)\,\psi(\D t).
\end{equation}
Therefore,
$$ \pi_m g =\sum_{j=1}^{N(d,m)} (g,Y_{mj})Y_{mj} = \beta_m[\Theta] \sum_{j=1}^{N(d,m)} (w_{K_2},Y_{mj})Y_{mj} = \beta_m[\Theta] \pi_m w_{K_2}.$$

We note that $\beta_m[\Theta]=0$ for odd $m$, since $\psi$ is an even measure and the Legendre polynomial $P_m(d;\cdot)$ is an odd function for odd $m$. This finally gives the first part of the following theorem.

\begin{theorem}\label{T4.1}
If $K_1,K_2\in\K^d$ and $\Theta\in{\mathcal M}_{d-1}$, then
\begin{equation}\label{4.15}
I(K_1,K_2,\Theta) = \sum_{m=0,\,m \text{ even}}^\infty \beta_m[\Theta]\left(\pi_m w_{K_1},\pi_m w_{K_2}\right),
\end{equation}
where $\beta_m[\Theta]$ is given by $(\ref{4.14})$.

If $K_1,K_2\in\K^d$ and $\Theta\in{\mathcal M}_1$, then
\begin{equation}\label{4.16}
J(K_1,K_2,\Theta) = \sum_{m=0,\,m \text{ even}}^\infty \beta_m[\Theta]\left(\pi_m b_{K_1},\pi_m b_{K_2}\right),
\end{equation}
where $b_{K_i}$ is the brightness function of $K_i$.
\end{theorem}

To prove the second part of this theorem, we note that a line $G\subset\R^d$ parallel to the unit vector $u$ can uniquely be written in the form $G= {\rm lin}\{u\}+y$ with $y\in u^\perp$. For $G$ represented in this way, we write $u=u(G)$. Let $K_i\in\K^d$ ($i=1,2$) and $\Theta\in{\mathcal M}_1$. We have
$$ J(K_1,K_2,\Theta) =\int_{G(d,1)^2} b_{K_1}(u(G_1))b_{K_2}(u(G_2))\,\Theta_0(\D(G_1,G_2))$$
by Lemma \ref{L2.1}, where $\Theta_0$ is a finite, jointly ${\sf G}(d)$-invariant and symmetric measure on $G(d,1)^2$.
Now the proof of the second part of Theorem \ref{T4.1} can be completed in the same way as that of the first part, just replacing the even function $w_{K_i}$ by the even function $b_{K_i}$.

\vspace{3mm}

\noindent{\em Proof of Theorem} \ref{T1.1c}.

Assume that $K\in\K^d$ is a convex body which satisfies $(\ref{1.4a})$ for each $\Theta\in{\mathcal M}_{d-1}$. We need only consider special measures $\Theta$, of the form
$$ \Theta = \int_{A(d,d-1)^2}{\mathbbm 1}\{(H_1,H_2)\in\cdot\}F(|\langle u(H_1),u(H_2)\rangle|)\,\mu_{d-1}^2(\D(H_1,H_2))$$
with a nonnegative, continuous function $F$. Since both sides of $(\ref{1.4a})$ are linear with respect to $\Theta$ (and hence $F$), it follows that $(\ref{1.4a})$ holds for any continuous function $F$. For such a function $F$, one obtains with the Funk--Hecke formula (\cite[p. 30]{Mul98}) that
$$ \beta_m[\Theta] =\omega_{d-1}\int_{-1}^1 F(|t|)P_m(d;t)(1-t^2)^{\frac{d-3}{2}}\D t,$$
where $\omega_{d-1}$ is the total spherical Lebesgue measure of the $(d-2)$-dimensional unit sphere.

Now let $k\in\N$ be even, $k\not=0$. Let $F$ be the restriction of the Legendre polynomial $P_k(d;\cdot)$ to $[-1,1]$. Then $F$ is an even function, and by the orthogonality properties of the Legendre polynomials (see \cite[p. 22]{Mul98}), saying that
$$ \int_{-1}^1 P_k(d;t)P_m(d;t)(1-t^2)^{\frac{d-3}{2}}\,\D t \left\{\begin{array}{lll} =0 & \mbox{if} & m\not=k,\\
\not=0 & \mbox {if} & m=k,\end{array}\right.$$
we have $\beta_m[\Theta]=0$ for $m\not= k$ and $\beta_k[\Theta]\not=0$. Therefore, $(\ref{1.4})$ (where now $\lambda[\Theta]=\beta_0[\Theta]\omega_d=0$ by (\ref{4.15})) and (\ref{4.15}) give $\pi_k w_K=0$ for $k\not= 0$ (note that (\ref{4.15}) can be applied, since $\Theta\in{\mathcal M}_{d-1}$). Since $w_K$ is an even function, we also have $\pi_kw_K=0$ for all odd $k$. Now the completeness of the system of spherical harmonics yields that $w_K$ is constant. This completes the proof of Theorem \ref{T1.1c}. \hspace*{\fill}$\Box$

It is clear that Theorem \ref{T1.2c} can be proved similarly.

\vspace{3mm}

\noindent{\em Proof of Theorem} \ref{T1.3}.

Suppose that $K_1$ is of constant width. Then the function $w_{K_1}$ is constant and hence $\pi_m w_{K_1}=0$ for $m\not=0$ (since constant functions are spherical harmonics of order $0$, and spherical harmonics of different orders are orthogonal). It follows from (\ref{4.15}) that
$$ I(K_1,K_2,\Theta)=  \beta_0[\Theta]\omega_d W(K_1)W(K_2)=\lambda[\Theta]  W(K_1)W(K_2).$$
Here we have used that $\pi_0 f= \omega_d^{-1}\int f\,\D\sigma$.

\noindent Authors' addresses:\\[2mm]
Daniel Hug\\Karlsruhe Institute of Technology, Department of Mathematics\\D-76128 Karlsruhe, Germany\\E-mail: daniel.hug@kit.edu\\[2mm]
Rolf Schneider\\Mathematisches Institut, Albert-Ludwigs-Universit{\"a}t\\D-79104 Freiburg i. Br., Germany\\E-mail: rolf.schneider@math.uni-freiburg.de

\end{document}